\documentclass[11pt,leqno,a4paper,english,leqno]{amsart}

\usepackage[left=3.5cm,right=3.5cm,top=3cm,bottom=3cm]{geometry}
\usepackage{amssymb,amsmath,amsfonts,amsthm,amscd}
\usepackage{leftidx}
\usepackage{tikz}
\usepackage{mathrsfs}
\usepackage{hyperref}
\hypersetup{
    colorlinks=true,       % false: boxed links; true: colored links
    linkcolor=blue,          % color of internal links
    citecolor=blue,        % color of links to bibliography
    filecolor=magenta,      % color of file links
    urlcolor=cyan           % color of external links
}

\usepackage{enumerate}

\newcommand{\n}[2]{\Sigma^{#1}_{#2}}

\newcommand{\eps}{\varepsilon}

\newcommand{\black}[1]{{\color{black}#1}}

\newtheorem{theorem}{Theorem}%[section]
\newtheorem{proposition}[theorem]{Proposition}
\newtheorem{lemma}[theorem]{Lemma}
\newtheorem{definition}[theorem]{Definition}
\newtheorem{example}[theorem]{Example}
\newtheorem{remark}[theorem]{Remark}
\newtheorem{corollary}[theorem]{Corollary}
\newtheorem{notation}[theorem]{Notation}

\newtheorem*{theo*}{Theorem}

\DeclareMathOperator{\supp}{supp}

\def\Id{\text{Id}}

\def\eps{\varepsilon}

\def\la{\left\lvert}
\def\lA{\left\lVert}
\def\ra{\right\rvert}
\def\rA{\right\rVert}

\def\blA{\big\lVert}

\def\brA{\big\rVert}

\def\les{\lesssim}

\def\ba{\begin{align}}
\def\bad{\begin{aligned}}
\def\be{\begin{equation}}
\def\ea{\end{align}}
\def\ead{\end{aligned}}
\def\ee{\end{equation}}
\def\e{\eqref}

\def\xK{\mathbb{K}}

\def\xC{\mathbb{C}}
\def\xR{\mathbb{R}}
\def\xN{\mathbb{N}}

\def\var{f}

\numberwithin{equation}{section}
\def\Sr{\mathcal{S}}

\def\mez{\frac{1}{2}}

\begin{document}

\title{Nonlinear interpolation and the flow map of quasilinear equations}

\author{Thomas Alazard}
\address{CNRS - Centre de Mathématiques Laurent Schwartz, {\'E}cole Polytechnique, 
Institut Polytechnique de Paris}
\email{thomas.alazard@polytechnique.edu}

\author{Nicolas Burq}
\address{Laboratoire de Math{\'e}matiques d’Orsay, Universit{\'e} Paris-Sud, Universit{\'e} Paris- Saclay}
\email{nicolas.burq@universite-paris-saclay.fr}

\author{Mihaela Ifrim}
\address{Department of Mathematics, University of Wisconsin, Madison}
\email{ifrim@wisc.edu}

\author{ Daniel Tataru}
\address{Department of Mathematics, University of California at Berkeley}
\email{tataru@math.berkeley.edu}

\author{Claude Zuily}
\address{Laboratoire de Math{\'e}matiques d’Orsay, Universit{\'e} Paris-Sud, Universit{\'e} Paris- Saclay}
\email{claude.zuily@universite-paris-saclay.fr}

\maketitle

\begin{abstract}
We prove an interpolation theorem for nonlinear functionals defined on scales of Banach spaces that generalize Besov spaces. It applies to functionals defined only locally, 
requiring only some weak Lipschitz conditions, extending those introduced by Lions and Peetre. Our analysis is self-contained and independent of any previous results 
about interpolation theory. It depends solely on the concepts of Friedrichs' mollifiers, seen through the formalism introduced by Hamilton, combined with the 
frequency envelopes introduced by Tao and used recently by two of the authors \text{and others} to study the Cauchy problem for various quasilinear 
\text{evolutions in partial differential} equations. Inspired by this latter work, our main application states that, for an abstract flow map of a 
quasilinear problem,  \black{both} the continuity of the flow  \black{ as a function of time and the continuity of the data-to-solution map} 
follow automatically from the estimates that are usually proven when establishing the existence of solutions: propagation of regularity via tame {\em a priori} estimates for higher regularities and contraction for weaker norms.

\end{abstract}

\section{Introduction}
The theory of interpolation originated from the study of Lebesgue spaces, with the Riesz-Thorin~\cite{MR1555250}  and Marcinkiewicz~\cite{marcinkiewicz1939interpolation} 
theorems, or Sobolev spaces with the study of traces of functions, pioneered by Lions~\cite{MR0119092}. An important generalization is the scale of Besov spaces $B^s_{p,q}(\mathbb{R}^d)$, 
which serves as the primary example studied in the classical books by Bergh and L\"{o}fstr\"{o}m~\cite{MR0482275} 
and Peetre~\cite{MR0461123}. The extension of interpolation theory to a nonlinear setting goes back to the \black{'70}, with works of  Lions~\cite{MR0267420}, Peetre~\cite{MR0482280},  Tartar~\cite{MR0310619}, 
Saut~\cite{MR0454374}, 
Bona and Scott~\cite{MR0393887}, and Maligranda~\cite{MR1024942}. However, it is fair to say that it has been 
relatively underexplored, 
particularly concerning applications that involve nonlinearities significant enough to be relevant for studying 
solutions to quasilinear partial differential equations. 
We will take advantage of this opportunity to introduce a new variation of the classical interpolation techniques, 
which enables us to derive automatic continuity results for 
nonlinear functions defined on generalized Besov spaces. Our analysis relies solely on the concepts of Friedrichs' mollifiers, seen through the formalism introduced by 
Hamilton~\cite{MR0656198}, combined with a simplified version of the notion of frequency envelopes introduced by Tao~\cite{MR1869874} for studying dispersive equations. 
In fact, our approach is strongly inspired by a series of recent papers ~\cite{MR4557379,MR4676214,MR4721644,ifrim2023sharp,pineau2023low, AIT} 
where the concept of frequency envelopes is used to offer an alternative perspective on the study of the Cauchy problem for quasilinear \black{partial differential } 
equations of different types, and in particular to prove continuous dependence of the solution as a function of the initial data.  

Our main result is given in the next section (see Theorem~\ref{Tmain}). 
We begin by stating its main corollary.
The latter is in the form of a general result which guarantees that the flow map of an abstract quasilinear equation exists 
and is continuous. More precisely, the next result asserts that the continuity of an abstract 
flow follows from the two estimates that are usually proven when establishing the existence of solutions to a quasilinear 
equation: (i) tame {\em a priori} estimate in \black{ a stronger} norm and (ii) \black{Lipschitz bounds} in a weaker norm. 
Here, using the terminology of Hamilton~\cite{MR0656198}, we say that an estimate is 
{\em tame} provided it depends linearly on the highest norm. We believe that this abstract result could be readily 
used to avoid 
some {\em ad-hoc} proofs that are somewhat subtle and technical, particularly those concerning the continuous 
dependence on initial data and the continuity in time of the solutions of quasilinear dispersive equations.

\begin{theorem}[Automatic continuity of the flow map]\label{princ}
Consider $\mu \in [2,+\infty]$ and four real numbers $s_0,s,s_1$ and   $r$  such that
$$
s_0<s<s_1,\quad \quad r>0.
$$
Consider $u_0$ in the Sobolev space $H^s=H^s(\xR^d)$ and a 
(possibly nonlinear) 
function defined on the ball of radius $r$ around $u_0$:
$$ \Phi\colon B_s(u_0,r) \mapsto L^\mu ((0,T); H^{s_0}),$$
where $B_s(u_0,r) = \{ v_0 \in H^s; \| v_0-u_0\|_{H^s} <r\}$, and satisfying the two following properties:
\begin{enumerate} [({H}1)]
\item\label{H1T} {\sc \black{Weak Lipschitz}}: there exists $C_0>0$ such that 
\begin{equation}\label{contract}
\forall v_0, w_0 \in B_s(u_0,r),\quad \lA \Phi(v_0) - \Phi(w_0)\rA_{L^\mu ((0,T); H^{s_0})} \le C_0\lA v_0 - w_0\rA_{H^{s_0}}.
\end{equation}
\item {\sc Tame estimate}: there exists $C_1>0$ such that (with $H^{+\infty} = \cap_s H^s$),
\begin{equation}\label{smooth}
\forall v_0 \in H^{+\infty} \cap B_s(u_0,r), \quad 
\lA \Phi(v_0) \rA _{L^\mu ((0,T); H^{s_1})} \le C_1 
\lA v_0\rA _{H^{s_1}}.
\end{equation}
\end{enumerate}

Then we have the three following conclusions:
\begin{enumerate}[i)]
\item {\sc Continuous interpolation}: for all $v_0 \in B_s(u_0, r)$, we have 
$$
\Phi(v_0)\in L^\mu ((0,T); H^s)$$
and the map 
$\Phi: v_0 \in B_s(u_0, r)\mapsto \Phi(v_0) \in L^\mu ((0,T); H^s)$ is continuous. 

\item {\sc Improved regularity in time}: 
for all $v_0 \in B_s(u_0, r)$, $\Phi(v_0)$ automatically belongs to a Chemin-Lerner space, which means that:
$$
\sum_{j\ge 0}\lA \Delta_j\Phi(v_0)\rA_{L^\mu((0,T);H^s)}^2<+\infty,
$$
where $u=\sum_{j\ge 0} \Delta_j u$ denotes the Littlewood-Paley decomposition \black{ of $u$} (see Appendix~\ref{appendix:LP}).

\item {\sc Continuity in time:} Assume that $\mu=+\infty$ and 
suppose in addition that 
for some $v_0\in B_s(u_0, r)$, we have 
\begin{equation}\label{N10}
\Phi(v_0) \in C^0([0,T]; H^{-\infty}).
\end{equation}
Then we also have 
\begin{equation}\label{N11}\Phi(v_0) \in C^0([0,T];H^s).
\end{equation}
\end{enumerate}
\end{theorem}

\begin{remark}
To avoid possible confusion in notations, let us emphasize that a key point is that 
the constants $C_0$ and $C_1$ may depend on the radius~$r$ and $u_0$. 
In particular, the estimate \e{contract} holds provided that there exists a non-decreasing function $\mathcal{F}_0\colon\xR_+\to\xR_+$ such that, for all $v_0, w_0$ in $H^{+\infty}\cap B_s(u_0,r)$,
$$
\lA \Phi(v_0) - \Phi(w_0)\rA_{L^\mu ((0,T); H^{s_0})} \le \mathcal{F}_0\big(\lA v_0\rA_{H^s}+\lA w_0\rA_{H^{s}}\big)\lA v_0 - w_0\rA_{H^{s_0}}.
$$
Similarly, the estimate \e{smooth} holds provided that there exists a non-decreasing function $\mathcal{F}_1\colon\xR_+\to\xR_+$ such that, for all $v_0$ in $H^{+\infty} \cap B_s(u_0,r)$,
$$
\lA \Phi(v_0) \rA _{L^\mu ((0,T); H^{s_1})} \le \mathcal{F}_1\big(\lA v_0\rA_{H^s}\big)
\lA v_0\rA _{H^{s_1}}.
$$
\end{remark}
\begin{remark}
If one knows that $ \Phi$ is continuous from $B_s(u_0,r)$ to 
$L^\mu ((0,T); H^{s_0})$, then it is enough to assume that the 
\black{weak Lipschitz} Assumption~$\ref{H1T}$ holds for smooth initial data, namely that there exists $C_0>0$ such that
\begin{equation}\label{contract3}
\forall v_0, w_0 \in H^{+\infty}\cap B_s(u_0,r),\quad \lA \Phi(v_0) - \Phi(w_0)\rA_{L^\mu ((0,T); H^{s_0})} \le C_0\lA v_0 - w_0\rA_{H^{s_0}}.
\end{equation}
\end{remark}
\begin{remark}
We shall prove a stronger result, which allows for instance to work in Besov spaces, 
and also to consider periodic functions or functions defined on manifolds, as well as other variants. 
\end{remark}
\begin{remark} 
This extends previous results (see in particular Saut~\cite{MR0454374} and Bona and Scott~\cite{MR0393887}) 
in several directions. The key point is that we only assume that $\Phi$ is defined locally. 
The idea is that this setting allows us to consider the case where 
both the time of existence $T$ and the constants $C_0, C_1$ depend on 
the $H^s$ size of the solutions. This is exactly the situation for the flow map of 
a quasilinear equation. Previous results made extra assumptions 
and also assumed that the flow was defined globally and that the 
constants $C_0$ and $C_1$ were uniform or depended on norms in a way that, 
in practical terms, required $\Phi$ to be the flow 
of a semilinear evolution equation (see Remark~\ref{R:2.11} for more details). 
By contrast, the assumptions in Theorem~\ref{princ} readily 
apply to quasilinear systems of evolution equations 
that can be solved by classical iterative schemes: 
the main example is the case of quasilinear systems 
of hyperbolic equations, but one can also consider quasilinear Schrödinger 
flows or water-wave-type systems, for which it is not at all trivial 
to prove continuous dependence with respect to the initial data. 
The second part of our conclusion was classically proved 
by repeating the Sobolev energy estimate (see for instance the proof of Lemma~$4.14$ in \cite{MR2768550}), 
and is quite cumbersome in practice. The above result asserts 
that this holds automatically. To the best of our knowledge, 
for instance in water-wave problems, this property 
was only recently proved and only for some very specific models, see e.g. \cite{AIT}.
\end{remark}

\black{
\begin{remark}
As a corollary of our proofs, one may easily extend the conclusion of the theorem to the full range of intermediate spaces $H^\sigma$ for $s < \sigma < s_1$. Precisely, under the same hypothesis, the conclusion of the theorem holds with $H^s$ replaced by $H^\sigma$, for all initial data $u_0 \in B_s(u_0,r) \cap H^\sigma$. 
\end{remark}
}

We conclude this introduction by discussing closely related 
problems where the main ideas used in the proof of Theorem~\ref{princ} have been introduced: 
the study of the Cauchy problem for hyperbolic equations and the proof of the paralinearization formula for nonlinear expressions.

The study of the well-posedness of the Cauchy problem for quasilinear equations began with the analysis of hyperbolic equations, with  
the celebrated works by Hadamard~\cite{MR0051411}, 
Sobolev~\cite{sobolev1935probleme,Sobolev1938theorem}, Courant, Friedrichs and Lewy \cite{MR1512478}, 
Schauder~\cite{schauder1935anfangswertproblem} or Petrowsky~\cite{petrowsky1937uber} (see the 
book by Benzoni-Gavage and Serre~\cite{MR2284507} for many results and references).  \black{ Kato's work \cite{Kato}, providing the first proofs of results of this type for quasilinear hyperbolic flows, has also been quite influential in this context.}
Many qualitative properties have been  deduced from the identity of weak and strong extensions of differential operators (see~\cite{MR0009701}), also known as Friedrichs' famous {\em weak=strong lemma} (loosely speaking,  
this lemma asserts that weak solutions 
obtained by duality are strong solution). It is from the latter perspective that we shall approach 
interpolation theory for nonlinear functions. In particular, the main technical ingredient of the proof (see Proposition~\ref{L2.13}) can be seen as a {\em weak=strong lemma} for the flow-map of an evolution equation. In addition, the first key inequality (see Lemma~\ref{L:smoothing}) is a simple refinement of the operator norm estimate of the Friedrichs muliplier.

The second point we want to emphasize is that the proof is significantly influenced by the arguments used in the proof of the so-called Meyer multiplier lemma (see the original article by Meyer~\cite{MR0639462} or the book by Alinhac and Gérard~\cite{MR2304160}), which is a classical lemma when proving Bony's paralinearization formula (see~\cite{MR0631751, MR0814548, MR2418072,alazard2023kam} for the composition $F\circ u$ of a smooth function $F$ and a function $u\in H^s(\xR^d)$ for some $s>d/2$). More precisely, the telescopic sum argument as well as the division between high and low frequencies in the proof of Lemma~\ref{L:expdecay} below are reminiscent of the proof of Meyer’s lemma. \black{ In this context one should also note the option of replacing the classical, discrete Littlewod-Paley decomposition with its continuous version, see \cite{tat-wm}, an idea which simplifies the bookkeeping by replacing telescoping sums by integrals, and allows for easier reiterations, particularly in geometric contexts.}

\medskip

\noindent\textbf{Plan of the paper.} 
In Section~\ref{S:2}, we introduce generalized Besov spaces and state our main interpolation result (see Theorem~\ref{Tmain}), from which we will \black{derive} Theorem~\ref{princ}. Theorem~\ref{Tmain} is then proved in Section~\ref{S:3}. 
We then show in Section~\ref{S:4} that Theorem~\ref{princ} can be deduced from our abstract interpolation result. In the appendix, we recall various classical notations and results about the Littlewood--Paley decomposition.

\section{Main result}\label{S:2}

\subsection{Generalized Besov spaces} Our aim will be to extend 
interpolation theory to nonlinear functions acting on scales of spaces that extend beyond Besov spaces (our interest lies more in the convenience of notation rather than strict generalization). To define such generalized Besov scales, three components are required in analogy with Besov spaces: a base space $E$, an index $s\in\xR$ giving the order of smoothness, and an index $q \in [1,+\infty]$ which represents a second-order correction to this order of smoothness. For instance, the Besov space $B^s_{p,q}(\xR^d)$ corresponds to $E=L^p(\xR^d)$, $s \in \mathbb{R}$, and $q \in [1,+\infty]$.

We will prove that, for our purpose, one can replace $L^p(\xR^d)$ by any vector space, not necessarily a Banach space, \black{equipped} 
with a pseudo-norm defined as follows:

\begin{definition}

Consider a vector space $E$ over $\xK=\xR$ or $\xK=\xC$, and a \black{functional} $N\colon E\to \xR_+$. 
We say that $N$ is symmetric, subadditive and point-separating if it satisfies the three following properties:
\begin{enumerate}
\item for all $x\in E$, $N(-x)=N(x)$,
\item $N(x+y)\le  N(x)+  N(y)$, for all $(x,y)\in E\times E$,
\item $N(x) =0$ if and only if $x=0$.
\end{enumerate}
Hereafter, we shall say that such an application $N$ is a pseudo-norm and that $(E,N)$ is a pseudo-normed vector space. We shall denote pseudo-norms 
as usual norms by~$\lA\cdot\rA_E$.
\end{definition}

Several explanations are in order concerning the above definition.

\begin{remark}\label{R1}

$(i)$ Consider a pseudo-normed vector space $(E,\lA\cdot\rA_E)$. 
Then, the \black{functional $d(x,y)=\lA x-y\rA_E$ on $E \times E$}  is a distance on $E$. 
In the following, it is always implicitly assumed that 
$E$ is endowed with the topology induced by this distance. 
%If $(E,d)$ is complete, then we say that $E$ is a pseudo-Banach space.

$(ii)$ Every norm is a pseudo-norm. 

$(iii)$ One motivation to consider pseudo-norms instead of norms 
comes from the study of local spaces. 
As an example, consider the vector space $L^2_{\rm loc}(\Omega)$ 
where $\Omega\subset\xR^d$ is an open set. 
Denote by $(K_n)_{n\in\xN}$ an exhaustion by compact sets of $\Omega$. Then 
$$
\lA f\rA_{L^2_{\rm loc}(\Omega)}=\sum_{n\in\xN}2^{-n}\frac{\lA f\rA_{L^2(K_n)}}{1+\lA f\rA_{L^2(K_n)}}
$$
is a pseudo-norm on $L^2_{\rm loc}(\Omega)$.

$(iv)$ One can generalize the previous example to include most of the natural Fr\'echet spaces: using the terminology introduced by Hamilton~\cite{MR0656198}, any topological vector space~$E$ equipped with a graded point-separating family of semi-norms $\rho_1\le \rho_2\le \cdots\le \rho_n\le \cdots$ can be seen as a pseudo-vector space for the pseudo-norm
$$
\lA f\rA_{E}=
\sum_{n\in\xN}2^{-n}\frac{\rho_n(f)}{1+\rho_n(f)}\cdot
$$

\end{remark}

Starting with a pseudo-normed vector space $E$,
the next step is to define the associated spaces of sequences.

\begin{definition}[Spaces of exponentially decaying sequences]
Consider a pseudo-normed 
vector space $(E,\lA \cdot\rA_E)$.

$(i)$ Given $s\in \xR$, we denote by $\Sigma^s_\infty(E)$ the vector space 
of those sequences $f=(f_k)_{k\in\mathbb{N}}$ with $f_k\in E$, such that
\be\label{N1}
\lim_{k\to+\infty} 2^{ks}\lA f_k\rA_E=0.
\ee

$(ii)$ Given $s\in\xR$ and $q\in [1,+\infty)$, 
we denote by $\Sigma^s_q(E)$ the vector space 
of those sequences $f=(f_k)_{k\in\mathbb{N}}$ with $f_k\in E$, such that 
\be\label{Np}
\sum_{k\in\xN} 2^{qks}\lA f_k\rA_E^q<+\infty.
\ee
\end{definition}

\black{
The quantities in \eqref{N1}, \eqref{Np} are akin to superimposing a weighted $\ell^q$ structure on 
the space of $E$ valued sequences. For convenience we recall and relate the appropriate 
notations.
}
\begin{notation}
Given $q\in [1,+\infty)$ (resp.\ $q=+\infty$), recall that 
 $\ell^q(\mathbb{N})$ denotes the space of those sequences $(u_k)_{k\in \mathbb{N}}$ such that 
$$
\lA u\rA_{\ell^q(\mathbb{N})}=\big(\sum_{k\in \mathbb{N}}\la u_k\ra^q\big)^{1/q}<+\infty\quad (\text{resp.}\quad 
\lA u\rA_{\ell^\infty(\mathbb{N})}=\sup_{k\in \mathbb{N}}\la u_k\ra<+\infty).
$$
Denote by $c_0(\xN)\subset \ell^\infty(\xN)$ the space of 
those sequences such that ${\displaystyle \lim_{k\to +\infty}\la u_k\ra=0}$. 

Then, given $q\in [1,+\infty)$ (resp.\ $q=+\infty$), 
$f\in \Sigma^s_q(E)$ if and only if $(2^{ks}\lA f_k\rA_E)_{k\in\xN}$ belongs to $\ell^q(\xN)$ (resp.\ $c_0(\xN)$).
\end{notation}

\black{
The $\Sigma^s_q(E)$ scale of spaces can be completed with appropriate end-points as follows:

}
\begin{notation}
Set
$$
\Sigma^\infty_1(E):=\bigcap_{s\in\xR}\Sigma^s_1(E)\quad;\quad \Sigma^{-\infty}_\infty(E):=\bigcup_{s\in\xR}\Sigma^s_\infty(E).
$$
Notice that, for all $s\in\xR$ and all $q\in [1,+\infty]$, there holds
$$
\Sigma^\infty_1(E)\subset \Sigma^s_1(E)\subset \Sigma^s_q(E)\subset \Sigma^s_\infty(E)\subset \Sigma^{-\infty}_\infty(E).
$$
\end{notation}

\black{Topologically, we can think of the $\Sigma^s_q(E)$ spaces as pseudo-normed spaces
as follows:
}
\begin{proposition}\label{P1}
Let $s\in \xR$ and $q\in [1,+\infty)$. Then the quantity 
$$
\lA f\rA_{\n{s}{q}}:= \left(\sum_{k\in\xN} 2^{qks}\lA f_k\rA_E^q\right)^\frac{1}{q}
$$
defines a pseudo-norm on the vector space $\Sigma^s_q(E)$. 
Similarly, the quantity
$$
\lA f\rA_{\n{s}{\infty}}:=\sup_{k\in\xN}2^{ks}\lA f_k\rA_E,
$$ 
defines a pseudo-norm on the vector space $\Sigma^s_\infty(E)$.
\end{proposition}
\begin{proof}
Let $q\in [1,+\infty]$. It is obvious that $\lA \cdot\rA_{\n{s}{q}}$ is symmetric and point-separating. Now, 
consider two sequences $f,g$ in $\Sigma^s_q(E)$. Since
$$
\lA f+g\rA_{\n{s}{q}}=\lA 2^{ks}\lA f_k+g_k\rA_E\rA_{\ell^q},
$$
we see that the triangle inequality for $\lA \cdot\rA_{\n{s}{q}}$ follows from the triangle inequality for the pseudo-norm $\lA \cdot\rA_E$ and for the classical 
norm $\lA \cdot\rA_{\ell^q(\xN)}$.
\end{proof}

\begin{remark}
If $\lA \cdot\rA_E$ is a norm, then $\lA \cdot\rA_{\n{s}{q}}$ is also a norm 
for all values of the parameters $s\in \xR$ and $q\in [1,+\infty]$. 
%In addition, $\Sigma^s_q(E)$ is a pseudo-Banach space if and only if $E$ is a pseudo-Banach space.
\end{remark}

\black{The elements of the $\Sigma^s_q(E)$ will play a role which is similar to the sequence of terms 
in the Littlewood-Paley decomposition of functions
in Besov spaces. In the broader context here, we 
define generalized Besov scales of spaces associated to the base space $E$ as follows:}

%\begin{definition}[Generalized Besov scale]\label{D3}
%Let $a,b\in \xR$, $q\in [1,+\infty]$ and 
%consider a pseudo-normed vector space $E$. 
%Let $(X^s,\lA \cdot\rA_{X^{s}})_{s\in [a,b]}$ be 
%a decreasing family of  \blue{pseudo-normed}  vector 
%spaces, satisfying
%$$
%s'\ge s\quad\Rightarrow\quad X^{s'}\subset X^s \quad \text{with}\quad 
%\lA \cdot \rA_{X^{s}}\le 
%\lA \cdot \rA_{X^{s'}}.
%$$
%We say that it is a scale 
%of pseudo-normed vector spaces of type $(q,E)$ 
%provided that there 
%exist two (possibly nonlinear) mappings 
%$$
%L\colon X^a\to \Sigma^a_q(E)\quad,\quad 
%R\colon \Sigma^a_q(E)\to X^a
%$$ 
%such that:
%\begin{enumerate}
%\item 
%$R(Lu)=u$ for all $u\in X^a$,
%\item $L(X^s)\subset \Sigma^s_q(E)$ and $R(\Sigma^s_q(E))\subset X^s$, 
%for all $s\in [a,b]$. Moreover, there exists a constant $C(s)>0$ such that
%\be\label{n160}
%\begin{aligned}
%&\lA L(u)-L(v)\rA_{\n{s}{q}}\le C(s)\lA u-v\rA_{X^s},\quad \forall u,v\in X^s,\\
%&\lA R(f)-R(f')\rA_{X^s}\le C(s)\lA f-f'\rA_{\n{s}{q}},\quad \forall f,f'\in \Sigma^s_q(E).
%\end{aligned}
%\ee
%\end{enumerate}
%\end{definition}

\begin{definition}[Generalized Besov scale]\label{D3}
Let $a\in\xR$, $q\in [1,+\infty]$ and 
consider a pseudo-normed vector space $E$. 
Let $(X^s,\lA \cdot\rA_{X^{s}})_{s\ge a}$ be 
a decreasing family of  \black{pseudo-normed}  vector 
spaces, that is a family satisfying
$$
s'\ge s\ge a\quad\Rightarrow\quad X^{s'}\subset X^s \quad \text{with}\quad 
\lA \cdot \rA_{X^{s}}\le 
\lA \cdot \rA_{X^{s'}}.
$$
We say that it is a scale 
of pseudo-normed vector spaces of type $(q,E)$ 
provided that there 
exist two linear mappings: 
$$
L\colon X^a\to \Sigma^a_q(E)\quad,\quad 
R\colon \Sigma^a_q(E)\to X^a
$$ 
such that the three following properties hold:
\begin{enumerate}
\item 
$R(Lu)=u$ for all $u\in X^a$,
\item for all $s\ge a$ and all $f\in \Sigma^s_1(E)$, there holds 
$L(Rf)\in \Sigma^s_1(E)$ together with the estimate
\be\label{N1017}
\lA L(Rf)\rA_{\Sigma^s_1}\le C(s)\lA f\rA_{\Sigma^s_1},
\ee
\item $L(X^s)\subset \Sigma^s_q(E)$ and $R(\Sigma^s_q(E))\subset X^s$, 
for all $s\ge a$. 
Moreover, there exists a constant $C(s)>0$ such that
\be\label{n160}
\begin{aligned}
&\lA Lu\rA_{\n{s}{q}}\le C(s)\lA u\rA_{X^s},\quad \forall u\in X^s,\\
&\lA Rf\rA_{X^s}\le C(s)\lA f\rA_{\n{s}{q}},\quad \forall f\in \Sigma^s_q(E).
\end{aligned}
\ee
\end{enumerate}
\end{definition}

\begin{notation}\label{N:norms}
Consider a scale $(X^s,\lA \cdot\rA_{X^{s}})_{s\ge a}$ 
of type $(q,E)$. Thanks to the correspondence with the space of sequences $\Sigma^s_q(E)$, we can define other functionals 
on $X^s$. Namely, we set
$$
\lA u\rA_{X^s_1}:=
\lA Lu\rA_{\n{s}{1}}\in [0,+\infty],\quad 
\lA u\rA_{X^s_\infty}:=
\lA Lu\rA_{\n{s}{\infty}}\in [0,+\infty).
$$
We also introduce the subspaces
$$
X^s_q:=R(\Sigma^s_q), \quad s\in (a,+\infty],\quad q\in 
[1,+\infty].
$$
\end{notation}
\begin{remark}
$i)$ Notice that we assume that $R\circ L=\Id$ but we 
do not assume that $L\circ R=\Id$. We only assume a weaker version of the latter property (see~\eqref{N1017}); 
this is important for applications. 

$ii)$ Notice that $\lA u\rA_{X^s_1}$ might be infinite for some $u\in X^s$. However, for any $q\in [1,+\infty]$, we have $\lA u\rA_{X^s_1}<+\infty$ for any 
$u\in X^{s'}$ for some $s'>s$. Indeed, 
$Lu\in \Sigma^{s'}_q(E)\subset \Sigma^s_1(E)$.

$iii)$ Since our main theorem applies for nonlinear functionals, we could as well assume that the operators $L$ and $R$ are nonlinear. 
However, in order to keep the exposition streamlined, we refrain from doing so.
\end{remark}
%\red{ There are several things  that we do not understand here. If $L$
%is nonlinear, why are these pseudonorms ?
%Also the first expression might not be finite for all $u \in X^s$ ?
%Also, should $s$ be restricted to $(a,b)$ ? Or $(a,\infty]$ ? 
%
%Maybe define the other structures simply as metric spaces ?
%}

\black{We now describe several interesting examples 
of generalized Besov scales:}

\begin{example}\label{Examples}
\begin{enumerate}
\item The Sobolev spaces $X^s=H^s(\xR^d)$ with $s\ge s_0$ (where $s_0$ is any given real number) 
and more generally the 
Besov spaces $B^s_{p,q}(\xR^d)$ with $p\in [1,+\infty)$ and $q\in [1,+\infty]$ are of type $(q,L^p(E))$. In this case the operators $L$ and $R$ are given by the Littlewood-Paley 
decomposition~(see Appendix~\ref{appendix:LP}) whose definition is recalled here. Fix a function $\varphi\in C^\infty_0(\xR^n)$ 
with support in an annulus $\{1/2\le \la \xi\ra\le 2\}$, so that
$$
\sum_{j=1}^\infty\varphi(2^{-j}\xi)=1-\psi(\xi),
\quad \supp\psi\subset\{|\xi|\leq1\}.
$$
We can then decompose any tempered distribution $u$ on $\xR^d$ as $u=\Delta_{0}u+\sum_{j\ge1}\Delta_ju$ with
$$
\Delta_0 u:=\mathcal{F}^{-1}(\psi \mathcal{F} u),
\quad \Delta_j u:=\mathcal{F}^{-1}(\varphi(2^{-j}\cdot)\mathcal{F} u),
\quad j\geq1,
$$
where $\mathcal{F}\colon \mathcal{S}'(\xR^d)\to \mathcal{S}'(\xR^d)$ denotes the Fourier transform.

Then, for $(s,p,q)\in \xR\times [1,+\infty) \times [1,+\infty]$, the Besov norm is defined by 
$$
\lA u\rA_{B^s_{p,q}}:=\bigg(\sum_{j\ge 0}2^{qjs}\lA\Delta_j u\rA_{L^p}^q\bigg)^{\frac{1}{q}}.
$$
Then $(B^s_{p,q})_{s\in\xR}$ is a scale of type $(q,L^p(\xR^d))$ provided that 
the operator $L$ and $R$ are defined as follows. 
Firstly, one sets $L\colon u\mapsto (f_j)_{j\ge 0}$ with $f_j=\Delta_j u$. 
To define $R$, we consider an auxiliary function $\widetilde{\varphi}
\in C^\infty_0(\xR^n\setminus\{0\})$ such that 
$\widetilde{\varphi}(\xi)\varphi(\xi)=\varphi(\xi)$, and then we define the operator $R$ by 
$$
R((f_j)_{j\ge 0})=f_0+\sum_{j\ge 0}\widetilde{\varphi}(2^{-j}D_x)f_j.
$$
Notice that $R$ is well-defined  \black{by the convergence in norm } and that $R(Lu)=u$. 
We do not have $L(Rf)=f$ for a given sequence $f=(f_j)_{j\ge 0}$. However, it is easily verified that
$\lA L(Rf)\rA_{\Sigma^\sigma_r}\le C(\sigma,r)\lA f\rA_{\Sigma^\sigma_r}$ 
for all $\sigma\in\xR$ and all $r\in [1,+\infty]$.

\item Consider the scale of Zygmund spaces, which plays a key role in applications since it extends H\"older spaces. Recall that they are the Besov spaces $C^s_*(\xR^d)=B^s_{\infty,\infty}(\xR^d)$ where $\lA u\rA_{B^s_{\infty,\infty}}=\sup_{j\ge 0}2^{js}\lA \Delta_j u\rA_{L^\infty}$. This scale 
does not enter into the previous framework since the condition~\eqref{N1} would not be satisfied. Instead, one can consider the so-called little H\"older spaces, that is the closure of $C^\infty_0(\xR^d)$ in $C^s_*(\xR^d)$.

\item To prove Theorem~\ref{princ}, we will use 
these spaces with $E=L^\infty((0,T);L^2(\xR^d))$. In this case, we recover the Chemin-Lerner spaces.

\item We can consider spaces of periodic functions, and more generally spaces of functions defined on manifolds. One could also consider local spaces, functions defined on open sets with boundaries or analytic functions.
\end{enumerate}
\end{example}

\subsection{Nonlinear interpolation}
\black{
Now that we have the appropriate notion of generalized Besov scales, we are ready to state our main result.
}
\begin{theorem}\label{Tmain}
Let $E,F$ be two pseudo-vector spaces and let 
$a\in \xR$ and $q\in [1,+\infty]$. 
Consider two scales 
$(X^s,\lA \cdot\rA_{X^{s}})_{s\ge a}$ 
and $(Y^s,\lA \cdot\rA_{Y^{s}})_{s\ge a}$ of type $(q,E)$ and $(q,F)$, respectively. 
Let $s_0,s,s_1$ and $r$ be real numbers such that
$$
a\le s_0<s<s_1,\quad \quad r>0.
$$ 
Consider a mapping $\Phi$ 
defined on a ball of radius $r$ around the origin of $X^s$ with values in $Y^{s_0}$: 
$$
\Phi\colon  B_{s}(0,r) \to Y^{s_0},
$$
where $B_{s}(0,r) = \{ v \in X^s; 
\| v\|_{X^{s}} <r\}$, and satisfying the following two properties:
\begin{enumerate} 
\item Contraction: there exists $C_0>0$ such that,
\begin{equation}\label{ncontractm}
\forall v, w
\in B_{s}(0,r),\quad 
\| \Phi(v) - \Phi(w)\|_{Y^{s_0}_{\infty}} \leq 
C_0\| v - w\|_{X^{s_0}_1}.
\end{equation}
\item Tame estimate: there exists $C_1>0$ such that,
\begin{equation}\label{nsmoothm}
\forall v \in X^\infty_1
\cap B_{s}(0,r),\quad 
\| \Phi(v) \| _{Y^{s_1}_{\infty}} \leq C_1\| v\| _{X^{s_1}_1}.
\end{equation}
\end{enumerate}
Then, we have the following two conclusions:
\begin{enumerate}
\item Interpolation: $\Phi(v)\in Y^s$ for all $v\in B_{s}(0,r)$;
\item Continuity: the mapping $\Phi\colon B_{s}(0, r)\mapsto Y^s$ is continuous. 
\end{enumerate}
\end{theorem}

\begin{remark}\label{R:2.11}
Several remarks are in order. 

Let us insist on the fact that the function $\Phi$ is not assumed to be linear. 
Compared to the nonlinear results 
by Tartar~\cite{MR0310619}, Bona and Scott~\cite{MR0393887} or Maligranda~\cite{MR1024942} the main differences are the following.
\begin{enumerate}
\item The assumptions of Theorem~\ref{Tmain} are local 
(the function $\Phi$ is assumed to be defined only on a neighborhood 
of the origin). This is essential since many nonlinear functionals are defined only on small balls (such as the flow map for a quasilinear evolution equation: indeed, for evolution equations, the 
flow map is defined on the whole space \black{if and only if} the Cauchy problem is well-posed globally in time). To prove such a local result necessitates a different proof \black{ compared to a corresponding global result.} Indeed, our proof is entirely distinct. It is direct, self-contained, and, in particular, does not rely on any results from the theory of interpolation. Of course, one can deduce a global result \black{ from our local result}.

Also, in the previous papers already mentioned, it is assumed that $\Phi$ satisfies the much stronger condition:
$$
\| \Phi(v) - \Phi(w)\|_{Y^{s_0}}
\le 
f(\lA v\rA_{X^{s_0}},\lA w\rA_{X^{s_0}})\| v - w\|_{X^{s_0}}.
$$
\black{However, for applications to quasilinear equations, this assumption is never satisfied. Instead, one proves 
$$
\| \Phi(v) - \Phi(w)\|_{Y^{s_0}} \leq 
f(\lA v\rA_{X^{s}},\lA w\rA_{X^{s}})\| v - w\|_{X^{s_0}},
$$
for some $s$ larger than $s_0$.  For instance in the case of nonlinear hyperbolic systems one typically chooses $s_0=0$ for an $L^2$-estimate, and $s>d/2+1$ so that the $H^s$-norm controls some Lipschitz norms of $v$ and $w$.}

\item We \black{only} assume weak bounds, in the sense that we estimate $\Phi(v)-\Phi(w)$ (resp.\  
$\Phi(v)$) in $Y^{s_0}_\infty$ (resp.\ $Y^{s_1}_\infty$). This \black{fact} will be used to prove that 
the flow-map in Theorem~\ref{princ} automatically satisfies Chemin-Lerner estimates.

\item We do not assume that $E$ or $F$ are normed vector spaces. Notice also that we do not assume that $E$ or $F$ are complete with respect to the natural distances mentioned in Remark~\ref{R1}.

\item The tame estimate~\e{nsmooth} is assumed to hold only for functions in the smallest set $X^\infty_1$. In practical terms, this implies that it suffices to establish a priori estimates for smooth functions.
\end{enumerate}
\end{remark}

\section{Proof of \black{the main result}}\label{S:3}
\black{ This section is devoted to the proof of the main result in Theorem~\ref{Tmain}.} From now on we fix three real numbers $s_0<s<s_1$ and 
$q\in[1,+\infty]$.

\subsection{Reduction to the model case}

Let $E,F$ be two pseudo-vector spaces and let 
$q\in [1,+\infty]$. 
Consider two scales \black{ of spaces} 
$(X^s,\lA \cdot\rA_{X^{s}})_{s\in [a,b]}$ 
and $(Y^s,\lA \cdot\rA_{Y^{s}})_{s\in [a,b]}$ of type $(q,E)$ and $(q,F)$, respectively. 

We start by observing that it is sufficient to prove Theorem~\ref{Tmain} in the special case where 
$$
(X^s,\lA \cdot\rA_{X^{s}})
=\big(\Sigma^s_q(E),\lA\cdot\rA_{\n{s}{q}(E)}\big),
\quad 
(Y^s,\lA \cdot\rA_{Y^{s}})=\big(\Sigma^s_q(F),\lA\cdot\rA_{\n{s}{q}(F)}\big).
$$
Indeed, it follows from Definition~\ref{D3} that 
there exists two pairs of linear mappings 
$(L_X,R_X)$ and $(L_Y,R_Y)$ such that
\begin{alignat*}{2}
&L_X\colon X^a\to \Sigma^a_q(E)\quad, \quad &&R_X\colon \Sigma^a_q(E)\to X^a,\\
&L_Y\colon Y^a\to \Sigma^a_q(F)\quad, \quad &&R_Y\colon \Sigma^a_q(F)\to Y^a,
\end{alignat*}
and such that:
\begin{enumerate}
\item 
$R_X(L_X(u))=u$ for all $u\in X^a$,
\item 
$R_Y(L_Y(u))=u$ for all $u\in Y^a$,
\end{enumerate}
together with the estimates \e{N1017} and~\e{n160}. 
Now, introduce the (nonlinear) \black{operator} 
\[
\Phi_\Sigma\colon \Sigma^s_q(E)\to \Sigma^{s_0}_q(F)
\]
defined by
\[
\Phi_\Sigma(f):=L_Y\big(\Phi(R_X(f))\big).
\]
Since $R_Y\circ L_y=\Id$ and similarly $R_X\circ L_X=\Id$, 
we see that  for $u \in X^s,$ 
$$
\Phi(u)=R_Y\big(\Phi_\Sigma(L_X(u))\big).
$$
By the continuity and regularity properties 
of the operators $L_X$, $R_X$, $L_Y$ and $R_Y$,
the two conclusions of Theorem~\ref{Tmain} for the operator $\Phi$ in the 
ball $B_s(u_0,r)$
are equivalent to the same conclusions applied to the operator $\Phi_{\Sigma}$ 
in the open set 
$R_X^{-1}(B_s(u_0,r)) \subset \Sigma^s_q(E)$. 
In a similar fashion, the properties in the hypothesis of Theorem~\ref{Tmain} 
for the operator $\Phi$ in $B_s(u_0,r)$ imply the same properties 
for $\Phi_{\Sigma}$ in 
$R_X^{-1}(B_s(u_0,r))$. Indeed, let us prove that $\Phi_\Sigma$ 
satisfies the following estimates:
\begin{align}
&\forall v, w
\in B_{s,q}(0,r),\quad 
\| \Phi(v) - \Phi(w)\|_{\n{s_0}{\infty}(F)} \leq 
C_0\| v - w\|_{\n{s_0}{1}(E)},\label{ncontract1}\\
&\forall v \in \Sigma^\infty_1(E) \cap B_{s,q}(0,r),\quad 
\| \Phi_\Sigma(v) \| _{\n{s_1}{\infty}(F)} \leq C_1\| v\| _{\n{s_1}{1}(E)}.\label{nsmooth1}
\end{align}
To prove \e{nsmooth1}, we write
\begin{alignat*}{2}
\| \Phi_\Sigma(v) \| _{\n{s_1}{\infty}(F)}
&=
\blA L_Y\big(\Phi(R_X(f))\big)\brA _{\n{s_1}{\infty}(F)}
&& (\text{by definition of $\Phi_\Sigma$})\\
&=\| \Phi(R_X(f)) \| _{Y^{s_1}_{\infty}} \qquad 
&&(\text{see Notation~\ref{N:norms}})\\
&\le C_1\| R_X(f)\| _{X^{s_1}_{1}}
&&(\text{by assumption~\eqref{nsmoothm}})\\
&\le C_1\| L_X(R_X(f))\|_{\Sigma^{s_1}_{1}(E)} \qquad 
&&(\text{see Notation~\ref{N:norms}})\\
&\les_{s_1} C_1 \| f\|_{\Sigma^{s_1}_{1}(E)} \qquad 
&&(\text{by assumption~\eqref{N1017}}),
\end{alignat*}
and we use similar arguments to get~\e{ncontract1}. 

Consequently, given an element $u \in B_s(u_0,r)$, it 
suffices to prove the result 
for $\Phi_\Sigma$ in a small ball around $L_X u$.
Namely, it follows that Theorem~\ref{Tmain} will be a consequence of the 
following result, which will be proved in this section.

\begin{theorem}\label{Tprinc}
Let $E$ and $F$ be two pseudo-normed vector spaces. 
Consider three real numbers $s_0<s<s_1$, 
together with $q\in [1,+\infty]$ and $r>0$. 
Consider a mapping $\Phi$ 
defined on a ball of radius $r$ around the origin of $\Sigma^s_q(E)$: 
$$
\Phi\colon  B_{s,q}(0,r) \to \Sigma^{s_0}_\infty(F),
$$
where $B_{s,q}(0,r) = \{ v \in \Sigma^s_q(E); \| v\|_{\n{s}{q}(E)} <r\}$. 

Assume that $\Phi$ satisfies the following two properties:
\begin{enumerate} 
\item  \black{ Weak Lipschitz}: there exists $C_0>0$ such that,
\begin{equation}\label{ncontract}
\forall v, w
\in B_{s,q}(0,r),\quad 
\| \Phi(v) - \Phi(w)\|_{\n{s_0}{\infty}(F)} \leq 
C_0\| v - w\|_{\n{s_0}{1}(E)}.
\end{equation}
\item Tame estimate: there exists $C_1>0$ such that,
\begin{equation}\label{nsmooth}
\forall v \in \Sigma^\infty_1(E) \cap B_{s,q}(0,r),\quad 
\| \Phi(v) \| _{\n{s_1}{\infty}(F)} \leq C_1\| v\| _{\n{s_1}{1}(E)}.
\end{equation}
\end{enumerate}
Then, we have the following two conclusions:
\begin{enumerate}
\item Interpolation: $\Phi(v)\in \Sigma^s_q(F)$ for all $v\in B_{s,q}(0,r)$;
\item Continuity: the mapping $\Phi\colon B_{s,q}(0, r)\mapsto \Sigma^s_q(F)$ is continuous. 
\end{enumerate}
\end{theorem}

\black{ 
\begin{remark}

We note here that the same arguments in the proof of the Theorem
also imply intermediate mapping properties for all indices $(\sigma,q_1)$ with
\[
s < \sigma < s_1, \qquad 1 \leq q_1 \leq  \infty.
\]
These are as follows:
\begin{enumerate}
\item Interpolation: $\Phi(v)\in \Sigma^\sigma_{q_1}(F)$ for all $v\in \Sigma^\sigma_{q_1}(E) \cap B_{s,q}(0,r)$;
\item Continuity: the mapping $\Phi\colon \Sigma^\sigma_{q_1}(E) \cap B_{s,q}(0, r)\mapsto \Sigma^\sigma_{q_1}(E)$ is continuous. 
\end{enumerate}
%\red{ We might also want a counterpart of this for Theorem~\ref{Tmain}.} 
\end{remark}
}

Notice that in the previous statement, we did not assume that $E$ or $F$ are normed 
vector spaces. Also, we did not assume that $E$ or $F$ are complete with respect 
to the natural distances mentioned in Remark~\ref{R1}. 
However, when $E$ and $F$ are Banach spaces, it is sufficient to assume that 
the contraction estimate holds for smooth functions.

\begin{proposition}\label{P19}
Assume that inTheorem~\ref{Tprinc},  $F$ is Banach spaces. 
Then the conclusions of Theorem~\ref{Tprinc} still hold if $\Phi$ satisfies~\e{nsmooth} 
together with the following assumption, which is weaker than~\e{ncontract}: 
there exists $C_0>0$ such that,
\begin{equation}\label{ncontractm2}
\forall v, w
\in \Sigma^\infty_1(E)\cap B_{s}(0,r),\quad 
\| \Phi(v) - \Phi(w)\|_{\Sigma^{s_0}_{\infty}(F)} \leq 
C_0\| v - w\|_{\Sigma^{s_0}_1(E)}.
\end{equation}
\end{proposition}
Proposition~\ref{P19} is proved in Section~\ref{S:3.6}.
\subsection{Reduction to the case $E=F$} 
We next prove that one can assume without loss of generality that $(E,\lA\cdot\rA_E)=(F,\lA\cdot\rA_F)$. Indeed, we can always reduce the general case to the latter by setting 
$$
B=E\times F\quad\text{equipped with the pseudo-norm}\quad \lA (u,v)\rA_B=\lA u\rA_E+\lA v\rA_F,
$$
and by considering the mapping $\tilde{\Phi}$, acting on sequences $f=(f_k)_{k\in\xN}$ with $f_k=(x_k,y_k)\in B=E
\times F$, by
$$
\tilde{\Phi}(f):=(0,\Phi(x)).
$$
We thus assume in the following that $(E,\lA\cdot\rA_E)=(F,\lA\cdot\rA_F)$ and write simply $\Sigma^s_q$ instead of $\Sigma^s_q(E)$.

\subsection{Friedrichs's mollifiers} 
We now introduce, following Hamilton~\cite{MR0656198}, a version of Friedrichs' mollifiers adapted to the space 
$\Sigma^s_q(E)$. 
Namely, for each $n\in\xN$, one introduces a 
linear map
$$
S_n\colon \Sigma^{-\infty}_\infty\to \Sigma^\infty_1,
$$
such that $S_n$ converges to the identity (in a sense to be made precise) when $n$ goes to $+\infty$. 

To do so, given $n\in\xN$ and a sequence $f=(f_k)_{k\in\xN}$ we define the sequence
\[
S_nf=((S_nf)_k)_{k\in\xN}
\] by
\[
(S_nf)_k:=f_k \quad\text{for}\quad k\le n\quad;\quad (S_nf)_k:=0 
\quad\text{for}\quad k> n.
\]
Since $(S_nf)_k$ vanishes for $k> n$, the sequence $S_nf$ obviously belongs 
to $\Sigma^s_1$ for all $s\in\xR$. The operators $S_n$ are thus smoothing operators. 
Let us recall two classical results.
\begin{lemma}\label{L:smoothing1}
$i)$ For all real numbers $r\le r'$, 
all $n\in\xN$ and all $q\in [1,+\infty]$, there holds
\be\label{n200}
\lA S_n f\rA_{\n{r'}{q}}\le 2^{n(r'-r)}\lA f\rA_{\n{r}{q}}.
\ee
$ii)$ For all $r\in \xR$, all $q\in [1,+\infty]$ and all $f\in \Sigma^r_q(E)$, 
there holds
$$
\lim_{n \to \infty} \lA f-S_n\var\rA_{\n{r}{q}}=0.
$$
\end{lemma}
\begin{proof}
$i)$ Directly from the definitions of $S_n$ and $\lA\cdot\rA_{\n{s}{q}}$, we have
\begin{align*}
\lA S_nf\rA_{\n{r'}{q}}^q=\sum_{k\le n} 2^{qkr'}\lA f_k\rA_E^q\le 
2^{qn(r'-r)}\sum_{k\le n} 2^{qkr}\lA f_k\rA_E^q
\le 2^{qn(r'-r)}\lA f\rA_{\n{r}{q}}^q,
\end{align*}
which implies the desired 
estimate.

$ii)$ When $q<+\infty$, the second statement is 
proved by writing
$$
\lA \var-S_n\var\rA_{\n{r}{q}}^q= \sum_{k\ge n+1}2^{qkr}\lA \var_k\rA_E^q,
$$
and then by observing that the right-hand side in the above inequality converges to $0$ as $n$ goes to $+\infty$ as the remainder of a convergent series. When $q=+\infty$, the results follows from the assumption~\eqref{N1}. 
\end{proof}

Assume that $q<+\infty$ and $r<r'$. The following observation contains a key improvement of \e{n200}. 
We claim that there exists a constant $K>0$  such that, for all $f\in \Sigma^{r}_q(E)$,
\be\label{n201}
\sum_{n\in\xN}2^{-qn(r'-r)}\lA S_n f\rA_{\n{r'}{q}}^q\le K \lA f\rA_{\n{r}{q}}^q.
\ee
To see this, write 
\begin{align*}
\sum_{n\in\xN}2^{-qn(r'-r)}\lA S_n f\rA_{\n{r'}{q}}^q&=\sum_{n\in\xN}2^{-qn(r'-r)}\bigg(\sum_{k=0}^n2^{qkr'}\lA f_k\rA_{E}^q\bigg)\\
&=\sum_{k\in\xN} \bigg(\sum_{n=k}^{+\infty}2^{-qn(r'-r)}\bigg)2^{qkr'}\lA f_k\rA_{E}^q\\
&\les_q \sum_{k\in\xN} 2^{-qk(r'-r)}2^{qkr'}\lA f_k\rA_{E}^q = \lA f\rA_{\n{r}{q}}^q.
\end{align*}
The next lemma further improves~\e{n201} by replacing 
$\lA S_n f\rA_{\n{r'}{q}}$ with the stronger norm $\lA S_n f\rA_{\n{r'}{1}}$. The proof relies on the Young's inequality for sequences, which we recall here to clarify notations. Given 
two sequences $u,v\colon \mathbb{Z}\to \xR$, we denote by $u*v$ the convolution product defined by
$$
(u*v)(n)=\sum_{p\in\mathbb{Z}}u(n-p)v(p).
$$
Given $q\in [1,+\infty)$ (resp.\ $q=+\infty$), 
denote by $\ell^q(\mathbb{Z})$ the space of those sequences $(u_k)_{k\in \mathbb{Z}}$ such that 
$\lA u\rA_{\ell^q}=\big(\sum_{k\in \mathbb{Z}}\la u(k)\ra^q\big)^{1/q}<+\infty$ (resp.\ $\lA u\rA_{\ell^\infty}=\sup_{k\in \mathbb{Z}}\la u(k)\ra<+\infty$). 
Denote by $c_0(\mathbb{Z})$ the sub-space of $\ell^\infty(\mathbb{Z})$ of 
those sequences $(u_k)_{k\in \mathbb{Z}}$ such that $\lim_{\la k\ra\to +\infty}\la u(k)\ra=0$. 
\black{ With these notations we have:} 
\begin{proposition}[Young's inequality]\label{P:Young}
Let $q\in [1,+\infty]$ and consider two sequences $u\in \ell^1(\mathbb{Z})$ and 
$v\in \ell^q(\mathbb{Z})$. Then $u*v\in \ell^q(\mathbb{Z})$ together with the estimate
$$
\lA u*v\rA_{\ell^q}\le \lA u\rA_{\ell^1}\lA v\rA_{\ell^q}.
$$
In addition, if $v\in c_0(\mathbb{Z})$ then $u*v\in c_0(\mathbb{Z})$.
\end{proposition}

We are now in position to prove the following

\begin{lemma}\label{L:smoothing}
$i)$ For all real numbers $r< r'$ 
and all 
$f\in \Sigma^r_\infty(E)$, there holds
$$
\sup_{n\in\xN}2^{-n(r'-r)}\lA S_n f\rA_{\n{r'}{1}}\le \frac{1}{1-2^{r-r'}} \lA f\rA_{\n{r}{\infty}}.
$$

$ii)$ For all real numbers $r< r'$, all $q\in [1,+\infty)$ and all 
$f\in \Sigma^r_q(E)$, there holds
$$
\Big(\sum_{n\in\xN}2^{-qn(r'-r)}\lA S_n f\rA_{\n{r'}{1}}^q\Big)^{\frac{1}{q}}\le \frac{1}{1-2^{r-r'}} \lA f\rA_{\n{r}{q}}.
$$
\end{lemma}
\begin{proof} 
Introduce the sequence $\gamma=(\gamma_n)_{n\in\xN}$ defined by 
$\gamma_n:=2^{-n(r'-r)}\lA S_n f\rA_{\n{r'}{1}}$. Note that
\begin{align*}
\gamma_n&=2^{-n(r'-r)}\sum_{k=0}^n2^{kr'}\lA f_k\rA_E
=\sum_{k=0}^n2^{(k-n)(r'-r)}2^{kr}\lA f_k\rA_E\\
&\le \sum_{k=0}^{\infty}2^{-\la k-n\ra (r'-r)}2^{kr}\lA f_k\rA_E.
\end{align*}
We seek to interpret the last expression as the convolution of two sequences $u,v\colon \mathbb{Z}\to \xR$. To do so, 
set
\begin{alignat*}{4}
u(p)&= 2^{-p(r'-r)} \quad &&\text{for}\quad p\in \xN, \quad &&u(p)= 0
\quad&&\text{for}\quad p\in \mathbb{Z}\setminus\xN,\\
v(p)&= 2^{pr}\lA f_p\rA_E \quad &&\text{for}\quad p\in\xN, \quad &&v(p)= 0
\quad&&\text{for}\quad p\in \mathbb{Z}\setminus\xN.
\end{alignat*}
Notice that $\lA u\rA_{\ell^1}=\sum_{p\in \xN}2^{-p(r'-r)}=1/(1-2^{r-r'})$. 
In addition, for all $n\in \xN$, we have
$$
\gamma_n\le (u*v)(n)=\sum_{p\in\mathbb{Z}}u(n-p)v(p).
$$
As a consequence, the result follows from Proposition~\ref{P:Young}.
\end{proof}

\subsection{Interpolation} 
In this \black{subsection} we will prove the key interpolation property \black{in Theorem~\ref{Tmain}}.

\begin{proposition}\label{L2.13}
Let $\var\in B_{s,q}(0,r)$. Then, under the assumptions of Theorem~\ref{Tprinc}, $\Phi(\var)$ belongs to $\Sigma^s_q(E)$. 
Moreover, 
$\Phi(S_n\var)$ is well-defined for all $n\in\xN$ 
and converges to $\Phi(\var)$ in $\Sigma^s_q(E)$ when 
$n$ tends to $+\infty$.
\end{proposition}
\begin{proof}
We begin \black{with} an elementary result which \black{asserts} that $\Phi(S_nf)$ is well-defined and converges to $\Phi(f)$ is 
the larger space $\Sigma^{s_0}_\infty(E)$.

\begin{lemma}\label{L:easyconv}
The sequence $\Phi(S_n\var)$ converges to $\Phi(\var)$ in 
$\Sigma^{s_0}_\infty$ when $n$ tends to $+\infty$.
\end{lemma}
\begin{proof}
Since $\lA S_n f\rA_{\n{s}{q}}\le \lA f\rA_{\n{s}{q}}$ by definition of $S_n$, 
we see that $S_n\var$ belongs to the ball $B_{s,q}(0,r)$ 
for all $n\in \xN$, and hence $\Phi(S_nf)$ is well-defined. 
Then the assumption~\e{ncontract} implies that
$$
\| \Phi(f) - \Phi(S_nf)\|_{\n{s_0}{\infty}} \leq 
C_0\| f - S_nf\|_{\n{s_0}{1}}.
$$
Since $s>s_0$, we have
$$
\| f - S_nf\|_{\n{s_0}{1}}\le \| f - S_nf\|_{\n{s}{q}},
$$
and hence the convergence of 
$\Phi(S_n\var)$ to $\Phi(\var)$ in 
$\Sigma^{s_0}_\infty$ follows directly from the fact that $f \in B_{s,q}(0,r)$ and the second statement in Lemma~\ref{L:smoothing1}.
\end{proof}

\black{In the next lemma we show how to obtain bounds for $\Phi(S_n f)$
in terms of similar bounds for $S_n f$.}

\begin{lemma}\label{L2}
Introduce the sequence 
$\gamma=(\gamma_n)_{n\in \xN}$ defined by 
\be\label{n0}
\gamma_n:=2^{-n(s_1-s)}\lA S_nf\rA_{\Sigma^{s_1}_1}.
\ee

$i)$ Then the mapping $f\mapsto \lA \gamma\rA_{\ell^q(\mathbb{N})}$ 
is a pseudo-norm on $\Sigma^s_q(E)$ equivalent to $\lA \cdot\rA_{\Sigma^s_q}$: 
\be
(1-2^{s-s_1})\lA \gamma\rA_{\ell^q(\mathbb{N})}\le \lA f\rA_{\n{s}{q}}\le \lA \gamma\rA_{\ell^q(\mathbb{N})},\qquad \forall q\in[1,+\infty].\label{n10}
\ee
Moreover, if $q=+\infty$, then $\lim_{n\to+\infty}\gamma_n=0$. 

$ii)$The following estimates hold :
\begin{align}
&\forall n\in\xN,\qquad \| \Phi(S_n\var)\|_{\n{s_1}{\infty}}
\le C_1 2^{n(s_1- s)} \gamma_n,\label{nhigh}\\[1ex]
&\forall n\in\xN,\qquad \|\Phi(S_{n+1}\var) - \Phi(S_n\var) \|_{\n{s_0}{\infty}} \le C_0 2^{-n (s-s_0)} \gamma_{n+1},\label{nlow}
\end{align}
where $C_0$ and $C_1$ are the constants  which appear 
in the assumptions~\eqref{ncontract} and~\eqref{nsmooth} on $\Phi$.
\end{lemma}
\begin{remark}
\black{We remark that the sequence $\gamma_n$ plays the role of a frequency envelope for $f$ in $\Sigma^s_q$, though with only a one sided slowly varying condition,
namely 
\[
\gamma_n \leq 2^{s_1-s} \gamma_{n+1}.
\]}
\end{remark}

\begin{proof}
$i)$ The first inequality in~\eqref{n10} follows from Lemma~\ref{L:smoothing} applied with 
$(r,r')=(s,s_1)$. The second one follows from the following obvious inequality
$$
\forall \ell\in\xN,\quad 2^{\ell s}\lA f_\ell\rA_E\le 2^{-\ell(s_1-s)}
\sum_{k=0}^\ell 2^{ks_1}\lA f_k\rA_E=\gamma_\ell.
$$
Moreover, if $q=+\infty$, the fact that $\lim_{n\to+\infty}\gamma_n=0$ follows from $\ell^1*c_0\subset c_0$ (see Proposition~\ref{P:Young}).  

$ii)$ To prove~\e{nhigh} and~\e{nlow}, we begin by verifying that 
\begin{align}
&\| S_n \var \|_{\n{s_1}{1}} = 2^{n(s_1- s)}\gamma_n,\label{n21}\\
&\| S_{n+1}\var - S_n\var\|_{\n{s_0}{1}}
\leq 2^{ -n(s- s_0) } \gamma_{n+1}.\label{n22}
\end{align}
The first equality follows from the very definition of $\gamma_n$. 
The second inequality is obtained by writing that
\begin{align*}
\| S_{n+1}\var - S_n\var \|_{\n{s_0}{1}}
&=2^{(n+1)s_0}\| \var_{n+1} \|_E
=2^{(n+1)(s_0-s)}2^{(n+1)s}\| \var_{n+1} \|_E\\
&\le 2^{(n+1)(s_0-s)} \gamma_{n+1}\le 
2^{-n(s- s_0) } \gamma_{n+1},
\end{align*}
where we used again the obvious bound $2^{\ell s}\lA f_\ell\rA_E\le \gamma_\ell$. Once~\e{n21} and~\e{n22} are established, the desired estimates \eqref{nhigh} and \eqref{nlow}  
follow directly from the assumptions~\eqref{nsmooth} and~\eqref{ncontract} on $\Phi$.
\end{proof}

The next result asserts that 
the coordinates of the 
sequence  
$\Phi(S_{n+1}\var) - \Phi(S_n\var)$ are 
decaying exponentially away from~$n$. 

\begin{lemma}\label{L:expdecay}
Set
$$
\kappa = \min (s_1-s, s-s_0)\quad\text{and}\quad
C=\max\{C_0,(1+2^{s_1-s})C_1\},
$$
where $C_0$ and $C_1$ are the constants  which appear 
in the assumptions~\eqref{ncontract} and~\eqref{nsmooth} on $\Phi$. 
Then, for all $n\in\xN$ and all $m\in\xN$, 
\be\label{n30}
2^{ms}\|(\Phi(S_{n+1}\var)- \Phi(S_n\var))_m\|_E
\le C 2^{-\kappa |m-n|} (\gamma_n+\gamma_{n+1}).
\ee
\end{lemma}
\begin{proof}
We consider the cases $m\le n$ and $m>n$ separately. 

To handle the case $m \leq n$, we write 
\begin{align*}
 &2^{ms}\| (\Phi(S_{n+1}\var) - \Phi(S_n\var))_m
 \|_E\\
 &\qquad\qquad= 
 2^{m(s-s_0)}2^{ms_0}\| (\Phi(S_{n+1}\var) - \Phi(S_n\var))_m\|_E \\
 &\qquad\qquad\le 
 2^{m(s-s_0)}\sup_{k\in\xN}2^{ks_0}\| (\Phi(S_{n+1}\var) - \Phi(S_n\var))_k\|_E \\ 
 &\qquad\qquad\le 2^{m(s-s_0) }\| \Phi(S_{n+1}\var) - \Phi(S_n\var) \|_{\n{s_0}{\infty}}
 \quad&&(\text{def. of }\lA\cdot\rA_{\n{s_0}{\infty}})\\
 &\qquad\qquad\le C_0 2^{(m-n) (s-s_0) } \gamma_{n+1} \quad &&(\text{cf.\ }\e{nlow})
\end{align*}
and hence the desired inequality \e{n30} is satisfied. 

It remains to consider the case $m > n$. 
For that purpose, we use the bound~\eqref{nhigh} to infer that
\begin{align*}
&2^{ms}\|(\Phi(S_{n+1}\var)-\Phi(S_n\var))_m\|_E\\
&\qquad= 2^{-m(s_1- s)}2^{ms_1}\|(\Phi(S_{n+1}\var)-\Phi(S_n\var))_m\|_E
\\
&\qquad\le 2^{-m(s_1- s)}\sup_{k\in\xN}2^{ks_1}\|(\Phi(S_{n+1}\var)-\Phi(S_n\var))_k\|_E
\\
&\qquad\le 2^{-m(s_1- s)}
\|\Phi(S_{n+1}\var)-\Phi(S_n\var)\|_{\n{s_1}{\infty}}\qquad&&(\text{def. of }\lA\cdot\rA_{\n{s_1}{\infty}})\\
&\qquad\le 2^{-m(s_1- s)}
\big(\lA\Phi(S_{n+1}\var)\rA_{\n{s_1}{\infty}}+\lA \Phi(S_n\var)\rA_{\n{s_1}{\infty}}\big)\qquad&&(\text{cf.\ Prop.\ } \ref{P1})\\
&\qquad\le C_1 2^{-m(s_1- s)} \big(2^{(n+1)(s_1-s)}\gamma_{n+1}+2^{n(s_1-s)}\gamma_n\big)\qquad &&(\text{cf.\ }\e{nhigh}),
\end{align*}
which shows that the desired inequality \e{n30} 
also holds when $m>n$.
\end{proof}

We are now ready to prove Proposition~\ref{L2.13}. The argument relies on a telescoping sum argument. Namely, notice that 
for each index $m\in\xN$ 
we have the following identity in the pseudo-normed vector space $E$:
\begin{equation*}
(\Phi(\var) - \Phi(S_n\var))_m = \sum_{p\ge n}\big(  \Phi(S_{p+1}\var) -  \Phi(S_p\var)\big)_m.
\end{equation*}
The previous summation is rigorously justified since Lemma~\ref{L:easyconv} insures 
that the sequence $(\Phi(S_n\var))_{n\in\xN}$ converges to $\Phi(f)$ in the space $\Sigma_\infty^{s_0}$ and hence, for all $m\in\xN$, there holds 
$\lim_{N\to+\infty}\lA (\Phi(f))_m-(\Phi(S_Nf))_m\rA_E=0$.

Introduce the sequence 
$c=(c_n)_{n\in \xN}$ defined by
\be\label{n1001}
c_n:=\gamma_n+\gamma_{n+1}.
\ee 
Now, it follows from Lemma~\ref{L:expdecay} 
that
\be\label{n140}
\begin{aligned}
2^{ms}\|  ( \Phi(\var) - \Phi(S_n\var))_m \|_E
&\le 2^{ms}\sum_{p\geq n}
\|(\Phi(S_{p+1}u)-\Phi(S_pu))_m
\|_E\\
&\le C \sum_{p\geq n} 2^{-\kappa |m-p|} c_p.
\end{aligned}
\ee

Let us now deduce from this that
\be\label{n1000}
\lim_{n\to+\infty}\|\Phi(\var) - \Phi(S_n\var) \|_{\n{s}{q}}=0.
\ee
First, we treat the case $q=+\infty$. Since the sequence $c=(c_p)_{p\in\xN}$ converges to $0$ when $p$ goes to $+\infty$, we easily verify that
$$
\lim_{m\to+\infty}\sum_{p\geq 0} 2^{-\kappa |m-p|} c_p=0,
$$
and so \e{n140} implies 
that $\Phi(\var) - \Phi(S_n\var)$ belongs to $\Sigma^s_\infty$ for all $n$. Let us prove that, in addition,  $\lim_{n\to+\infty}\|\Phi(\var) - \Phi(S_n\var) \|_{\n{s}{\infty}}=0$. To do so, consider $\epsilon>0$. Using again the fact that the sequence $c=(c_p)_{p\in\xN}$ converges to $0$ when $p$ goes to $+\infty$, we obtain that there exists $n_0\in \xN$ such that, for all $n\ge n_0$ and all $m\in\xN$,
$$
\sum_{p\geq n} 2^{-\kappa |m-p|} c_p\le \Big(\sum_{p\geq n} 2^{-\kappa |m-p|}\Big)\epsilon
\le \beta\epsilon \quad\text{where}\quad \beta=\sum_{b\in\mathbb{Z}}2^{-\kappa \la b\ra}=\frac{2}{1-2^{-\kappa}}.
$$

It remains to treat the case $q<+\infty$. 
Using  Young's inequality $\ell^1 *\ell^q\subset \ell^q$ 
(see Proposition~\ref{P:Young}), we deduce from~\e{n140} that
\be\label{n101}
\begin{aligned}
&\|\Phi(\var) - \Phi(S_n\var) \|_{\n{s}{q}}
\le 
AC\Big(\sum_{p\ge n}c_p^q\Big)^\frac{1}{q}\quad\text{where}\\
&A=\sup_{m\in\xN}\sum_{p\in\xN}2^{-\kappa\la p-m\ra}\le \sum_{p\in\mathbb{Z}}2^{-\kappa\la p\ra}=\frac{2}{1-2^{-\kappa}}\cdot
\end{aligned}
\ee 
This implies 
the desired result since $\sum_{p\ge n}c_p^q$ goes to $0$ when $n$ tends to $+\infty$, as $c\in \ell^q(\xN)$ (see~\e{n10}).

This concludes the proof of Proposition~\ref{L2.13}.
\end{proof}

\subsection{Continuity}

So far, we have proved that $\Phi(f)$ belongs to $\Sigma^s_q$ for any $f\in B_{s,q}(0,r)$. 
It remains to prove that the mapping $\Phi\colon B_{s,q}(0,r)\to \Sigma^s_q$ is continuous. 

Consider $f\in \Sigma^s_q$ and a sequence $(f_N)_{N\in\xN}$ converging to $f$ in $\Sigma^s_q$. We want to prove that $\Phi(f_N)$ converges to $\Phi(f)$ in $\Sigma^s_q$ as $N$ tends to $+\infty$. The proof is based on a $3\eps$-argument. Namely, it relies on the following triangle inequality, applied with $n$ large enough:
\be
\begin{aligned}\label{3eps}
\| \Phi(f_{N}) - f \|_{\n{s}{q}}
&\leq 
\| \Phi(f_N) - \Phi(S_nf_N)\|_{\n{s}{q}} \\
&\quad+  \| \Phi(S_nf_N) - \Phi(S_n\var)\|_{\n{s}{q}} \\
&\quad+  \| \Phi(S_n\var) - f\|_{\n{s}{q}}.
\end{aligned}
\ee
As a direct corollary of the inequality~\e{n101}, we will see 
that the first and the third terms go to $0$ when $n$ tends to $+\infty$, 
\black{uniformly in $N$}. Then, the second term will be analyzed 
by observing that, once $n$ is fixed, the term 
$\Phi(S_nf_N) - \Phi(S_n\var)$ goes to $0$ in $\n{s}{q}$ 
as $N$ tends to $0$ by an elementary interpolation argument (since this difference converges to $0$ in $\n{s_0}{\infty}$ and is bounded in $\n{s_1}{q}$).

Let us now proceed \black{with} the details. 
Introduce the sequences
$$
\gamma=(\gamma_{n})_{n\in\xN},\quad 
c=(c_{n})_{n\in\xN},\quad \gamma_N=(\gamma_{N,n})_{n\in\xN},\quad
c_N=(c_{N,n})_{n\in\xN}\quad\text{with}\quad N\in\xN,
$$
defined by
\begin{alignat*}{3}
&\gamma_{n}:=2^{-n(s_1-s)}\lA S_n f\rA_{\n{s_1}{1}}\quad&&\text{and}\quad
&&c_{n}:=\gamma_{n+1}+\gamma_{n},\\
&\gamma_{N,n}:=2^{-n(s_1-s)}\lA S_n f_N\rA_{\n{s_1}{1}}\quad&&\text{and}\quad
&&c_{N,n}:=\gamma_{N,n+1}+\gamma_{N,n}.
\end{alignat*}
It follows from the triangle inequality applied in $\Sigma^{s_1}_1$ that
$$
\la \gamma_{N,n}-\gamma_n\ra\le 2^{-n(s_1-s)}\lA S_n (f_N-f)\rA_{\n{s_1}{1}}.
$$
Therefore, 
Lemma~\ref{L:smoothing} implies that
$$
\lA \gamma_N-\gamma\rA_{\ell^q}\le \frac{1}{1-2^{s-s_1}}\lA f_N-f\rA_{\n{s}{q}},
$$
and hence 
$\gamma_N$ 
converges in $\ell^q(\xN)$ to $\gamma$ as $N$ tends to $+\infty$. Consequently, $c_N$ 
converges in $\ell^q(\xN)$ to $c$ as $N$ tends to $+\infty$. 

Now, notice that, since for $N$ large enough we have 
$f_{N}\in B_{s,q}(0,r)$, we can apply the estimate~\e{n101} with $\var$ replaced by $\var_{N}$, 
to write 
$$
\| \Phi(\var_N) - \Phi(S_n\var_N)\|_{\n{s}{q}} \les  \lA 1_{\ge n} c_{N}\rA_{\ell^q},
$$
where $1_{\ge n}=(1_{\ge n}(p))_{p\in\xN}$ is the sequence defined by $1_{\ge n}(p)=1$ provided that $p\ge n$, and $1_{\ge n}(p)=0$ otherwise. Notice that this inequality holds both for $q<+\infty$ and $q=+\infty$. 
So, by using the triangle inequality in $\ell^q(\xN)$, 
\be\label{n51}
\| \Phi(\var_N) - \Phi(S_n\var_N)\|_{\n{s}{q}} \les  \lA 1_{\ge n} c_{N}\rA_{\ell^q}\le \lA 1_{\ge n} c\rA_{\ell^q}+\lA  c_{N}-c\rA_{\ell^q}.
\ee

Let $\eps>0$. Since 
$1_{\ge n} c$ goes to 
$0$ in $\ell^q(\xN)$ as $n$ tends to $+\infty$ (as already seen) and since $c_N-c$ goes to $0$ in $\ell^q(\xN)$ as $N$ tends to $\infty$, one can find $n_0$ and $N_0$ such that,
$$
\| \Phi(\var_N) - \Phi(S_n\var_N)\|_{\n{s}{q}}\le \eps
$$
for all $n\ge n_0$ and all $N\ge N_0$. Moreover, 
since (see Proposition~\ref{L2.13})
$$
\lim_{n\to+\infty}\|\Phi(\var) - \Phi(S_n\var) \|_{\n{s}{q}}=0,
$$
one can further assume that
$$
\| \Phi(\var) - \Phi(S_n\var)\|_{\n{s}{q}}\le \eps,
$$
for all $n\ge n_0$. Now, write
\begin{align*}
\| \Phi(f_{N}) - f \|_{\n{s}{q}}
&\leq 
\| \Phi(f_N) - \Phi(S_nf_N)\|_{\n{s}{q}} \\
&\quad+  \| \Phi(S_nf_N) - \Phi(S_n\var)\|_{\n{s}{q}} \\
&\quad+  \| \Phi(S_n\var) - f\|_{\n{s}{q}}.
\end{align*}
When $n\ge n_0$ and $N\ge N_0$, the first and the last term in the right-hand side above are smaller than $\eps$. 
It remains only to estimate the second term. More precisely, to conclude the proof of Theorem~\ref{Tprinc}, it is sufficient to prove that, for all $n\in\xN$, there holds
\be\label{n120}
\lim_{N\rightarrow + \infty }
\| \Phi(S_nf_N) - \Phi(S_n\var)\|_{\n{s}{q}}=0.
\ee
To do so, we notice that, given some fixed $n\in\xN$, 
the set 
$$
\{ S_nf_N ; N\in\xN\}\cup \{S_n\var\}
$$
is bounded in $\Sigma^{s_1}_q$ 
(see Lemma~\ref{L:smoothing}). 
Therefore, the assumption~\e{nsmooth} implies that 
$$
\sup_{N\in\xN}\| \Phi(S_nf_N) - \Phi(S_n\var)\|_{\n{s_1}{\infty}}
\le \sup_{N\in\xN}
\Big(\| \Phi(S_nf_N)\|_{\n{s_1}{\infty}}+ \|\Phi(S_n\var)\|_{\n{s_1}{\infty}}\Big)<+\infty.
$$
On the other hand, the assumption~\e{ncontract} implies that
$$
\| \Phi(S_nf_N) - \Phi(S_n\var)\|_{\n{s_0}{\infty}}\le C_0\lA S_nf_N-S_n\var\rA_{\n{s_0}{1}}
\le C_0\lA f_N-\var\rA_{\n{s_0}{1}},
$$
where we used the trivial estimate $\lA S_n\rA_{\Sigma^{s_0}_1\to \Sigma^{s_0}_1}\le 1$ to get the last inequality. 
This proves that $\| \Phi(S_nf_N) - \Phi(S_n\var)\|_{\n{s_0}{\infty}}$ converges to $0$ as $N$ tends to $+\infty$. \black{Now we obtain the desired inequality \eqref{n120} using the 
following straightforward interpolation inequality.}

\begin{lemma}
For all $s\in (s_0,s_1)$ and all $q\in [1,+\infty]$, there exists a constant $C$ such that, 
for all $f\in \Sigma^{s_1}_\infty$,
\be\label{interp}
\| f\|_{\n{s}{q}} \le C \| f\|_{\n{s_0}{\infty}}^{\theta} \| f\|_{\n{s_1}{\infty}}^{1-\theta}, 
\qquad \text{with}\quad\theta = \frac{s_1-s}{s_1-s_0}.
\ee
\end{lemma}
\begin{proof}
Let $N\in\xN$ and consider the decomposition $f=S_Nf+(I-S_N)f$. 
Notice that
\begin{align*}
\| S_Nf\|_{\n{s}{q}}^q&=\sum_{n=0}^N2^{nqs}\lA f_n\rA_E^q=
\sum_{n=0}^N2^{nqs_0}2^{nq(s-s_0)}\lA f_n\rA_E^q\\
&\le 
\left(\sum_{n=0}^N2^{nq(s-s_0)}\right)\Big(\sup_{n\in\xN}2^{nqs_0}\lA f_n\rA_E^q\Big)^q\les 2^{Nq(s-s_0)}\lA f\rA_{\n{s_0}{\infty}}^q.
\end{align*}
Similarly, we have
\begin{equation*}
\| (I-S_N)f\|_{\n{s}{q}}^q
\le 
\left(\sum_{n=N+1}^\infty2^{nq(s-s_1)}\right)\Big(\sup_{n\in\xN}2^{nqs_1}\lA f_n\rA_E^q\Big)^q
\les 2^{Nq(s-s_1)}\lA f\rA_{\n{s_1}{\infty}}^q.
\end{equation*}
It then follows from the triangle inequality that
$$
\lA f\rA_{\n{s}{q}}\les 2^{N(s-s_0)}\lA f\rA_{\n{s_0}{\infty}}+2^{-N(s_1-s)}\lA f\rA_{\n{s_1}{\infty}},
$$
from which we obtain the desired inequality~\e{interp} by optimizing in 
$N$.
\end{proof}

The proof of 
Theorem \ref{Tprinc} is now complete. 

\subsection{Proof of Proposition~\ref{P19}}\label{S:3.6}

We now assume that $\Phi$ satisfies~\e{nsmooth} together with~\e{ncontractm2}. 
Recall that the latter is a weak form of assumption~\e{ncontract}, where the estimate is assumed 
to hold only for smooth elements $v,w$ in $\Sigma^\infty_1(E) \cap B_{s}(0,r)$. 
The only place where we used the stronger assumption~\e{ncontract} was in the proof 
of Lemma~\ref{L:easyconv}. This lemma asserts that the 
sequence $\Phi(S_n\var)$ converges to $\Phi(\var)$ in $\Sigma^{s_0}_\infty(F)$ as $n$ 
tends to $+\infty$. Therefore, we only need to prove that 
this result holds under the weaker assumption~\e{ncontractm2}, 
provided that $F$ is a Banach spaces. To do so, since $S_{k}f$ 
belongs to $\Sigma^\infty_1(E) \cap B_{s}(0,r)$ 
for all $k\in\xN$, one apply~\e{ncontractm2} to deduce that
$$
\sum_{p=n}^{+\infty}\| \Phi(S_{p+1}f) - \Phi(S_pf)\|_{\n{s_0}{\infty}} \le 
C_0\sum_{p=n}^{+\infty}\| S_{p+1}f - S_pf\|_{\n{s_0}{1}}.
$$
It follows that
\be\label{telecopic101}
\begin{aligned}
\sum_{p=n}^{+\infty}\| \Phi(S_{p+1}f) - \Phi(S_pf)\|_{\n{s_0}{\infty}} 
&\le C_0\sum_{k=n+1}^{+\infty} 2^{k s_0}\| f_{k}\|_{E} \\
&\les C_0\left(\sum_{k=n+1}^{+\infty} 2^{qk s}\| f_{k}\|_{E}^q\right)^{1/q}.
\end{aligned}
\ee

Now, observe that, since $\n{s_0}{\infty}(F)$ is a Banach space when $F$ is one, 
the fact that the right-hand side of \e{telecopic101} is finite implies that the series $\sum_{p=n}^{+\infty} \Phi(S_{p+1}f) - \Phi(S_pf)$ is normally convergent and hence converges in~$\n{s_0}{\infty}(F)$. Moreover, since it is assumed that the mapping $\Phi$ is continuous from 
the ball of radius $r$ around the origin of $\Sigma^s_q(E)$ with values in $\Sigma^{s_0}_q(F)$, it is also 
continuous with values in $\Sigma^{s_0}_\infty(F)$. This justifies writing:
$$
\Phi(f) - \Phi(S_nf) = \sum_{p=n}^{+\infty} \Phi(S_{p+1}f) - \Phi(S_pf).
$$
Secondly, since the right-hand side in \e{telecopic101} converges to $0$ as $n$ tends to infinity, 
we conclude that the sequence $\Phi(S_n\var)$ converges to $\Phi(\var)$ in $\Sigma^{s_0}_\infty(F)$ as $n$ tends to $+\infty$. This completes the proof.

\section{Proof of Theorem~\ref{princ}}\label{S:4}

\black{We show here} that Theorem~\ref{princ} can be deduced from our abstract interpolation 
\black{ result in Theorem~\ref{Tmain}}.
 One interesting feature of the proof is that we will exploit the fact that we only assume 
some weak bounds in Theorem~\ref{Tmain} (by weak bounds, we refer to the fact that the mapping $\Phi$ is only estimated with the norms $\lA\cdot\rA_{Y^{s_0}_\infty}$ and $\lA\cdot\rA_{Y^{s_1}_\infty}$ while the conclusion involves the strong norm $\lA\cdot\rA_{Y^s}$). 

Let us introduce several 
scales of Banach spaces. 
On the one hand the Sobolev scale $X^s=H^s(\xR^d)$, which is of type $(2,L^2(\xR^d))$ when the operators $L$ and $R$ are given by the Littlewood-Paley decomposition, as explained in Example~\ref{Examples}. 
We also introduce the scale of time-dependent 
functions $f=f(t)$ with values in $H^s$, which are $L^\mu$ in time, that is the spaces
$$
L^\mu((0,T);H^s(\xR^d)).
$$
We further introduce a third scale, which we refer to as 
the Chemin-Lerner 
scale $Y^s$, defined by: 
$$
Y^s=\left\{ f\in L^\mu((0,T);H^s(\xR^d))\,;\,
\lA f\rA_{Y^s}=\Big(\sum_{j\in \xN} \lA \Delta_j f\rA_{L^\mu([0,T];H^s(\xR^d))}^2\Big)^{\frac{1}{2}}<+\infty\right\},
$$
where $I=\sum_{j\in\xN}\Delta_j$ still denotes the Littlewood-Paley decompositon of tempered distributions $u\in\mathcal{S}'(\xR^d)$ and
$$
\lA \Delta_j f\rA_{L^\mu([0,T];H^s(\xR^d))}^2
=\bigg(\int_0^T \lA \Delta_j (f(t))\rA_{H^s}^\mu\, dt\bigg)^{\frac{2}{\mu}}.
$$
Now the two important observations are the following. Firstly, $Y^s$ is a scale of Banach spaces of type $(2,F)$ where $F=L^\mu((0,T);L^2(\xR^d))$. And secondly, 
since $\mu\ge 2$ by assumption,  
it follows from the Minkowski inequality (see Stein~\cite{Stein1970}) that 
$$
\lA f\rA_{L^\mu([0,T];H^s(\xR^d))}\le \lA f\rA_{Y^s}.
$$
There is in addition a little trick: 
since $\lA \Delta_j\rA_{H^r\to H^r}\le 1$ (as follows from Plancherel's theorem) for all $j\in\xN$ and all $r\in\xR$, we have
$$
\lA \Delta_j f\rA_{L^\mu([0,T];H^r(\xR^d))}
\le 
\lA f\rA_{L^\mu([0,T];H^r(\xR^d))}
$$
and hence
$$
\lA f\rA_{Y^r_\infty}=\sup_{j\in\xN}
\lA \Delta_j f\rA_{L^\mu([0,T];H^r(\xR^d))}
\le 
\lA f\rA_{L^\mu([0,T];H^r(\xR^d))}.
$$
Therefore, under the assumptions of Theorem~\ref{princ} we see that
$$
\| \Phi(v) - \Phi(w)\|_{Y^{s_0}_{\infty}} \leq 
\| \Phi(v) - \Phi(w)\|_{L^\mu((0,T);H^{s_0})
}
$$
and similarly 
$$
\| \Phi(v)\|_{Y^{s_1}_{\infty}} \leq 
\| \Phi(v)\|_{L^\mu((0,T);H^{s_1})
}.
$$
Therefore, under the assumption of Theorem~~\ref{princ}, the flow map $\Phi$ satisfies those of Theorem~\ref{Tmain} applied with $$
q=2,\quad E=L^2(\xR^d)\quad\text{and}\quad
F=L^\mu((0,T);L^2(\xR^d)).
$$
And hence we deduce the first two points of Theorem~\ref{princ}.

It remains to prove the last \black{assertion of Theorem~\ref{princ}} about the continuity in time. Consider a function $v_0\in B_s(u_0,r)$ such that $\Phi(v_0)$ belongs to $C^0([0,T];H^{-\infty}(\xR^d)))$. 
We want to prove that 
$\Phi(v_0)$ belongs to $C^0([0,T];H^{s}(\xR^d))$. 
Given an integer $N$, 
introduce the smoothing operator 
$S_N\colon H^{-\infty}(\xR^d)\to H^{+\infty}(\xR^d)$ defined by
$$
S_N u:=\sum_{j=0}^N\Delta_j u. 
$$
Now, the functions $S_N\Phi(v_0)$ defined by 
$$
(S_N(\Phi(v_0)))(t):=S_N(\Phi(v_0)(t))
$$
belong to $C^0([0,T];H^{+\infty}(\xR^d))$. 
To prove that $\Phi(v_0)$ belongs to $C^0([0,T];H^{s}(\xR^d))$, it is thus sufficient to show that $S_N\Phi(v_0)$ tends to $\Phi(v_0)$ in 
$L^\infty([0,T];H^s(\xR^d))$ when $N$ tends to $+\infty$. To do so, we 
use that, by almost orthogonality, we have
$$
\lA \Phi(v_0)-S_N\Phi(v_0)\rA_{L^\infty([0,T];H^s(\xR^d))}^2\le K \sum_{j\ge N-N_0}\lA \Delta_j \Phi(v_0)\rA_{L^\infty([0,T];H^s(\xR^d))}^2,
$$
for some constants $K$ and $N_0$ depending only on the  choice of 
functions $\psi$ in the Littlewood-Paley decomposition. Since the series in the right-hand is convergent (as $\Phi(v_0)\in Y^s$), we see that the right-hand side goes to $0$ as $N$ tends to $+\infty$ as the remainder of a converging series. 
This completes the proof.

\appendix 

\section{Littlewood-Paley decompositions}\label{appendix:LP}

We recall here some elementary results 
about the Littlewood--Paley decomposition. This decomposition makes it possible to introduce a parameter into a problem
that has none. 
There are many books that develop a systematic study
of the Littlewood--Paley decomposition: Alinhac and G\'erard~\cite{MR2304160}, Bahouri, Danchin and Chemin~\cite{MR2768550}, Coifman and Meyer~\cite{CoMeAs,MeCo}, M\'etivier~\cite{MR2418072}, 
Muscalu and Schlag~\cite{Muscalu-Schlag1}, Stein~\cite{Stein1970,Stein1993}, Tao~\cite{Tao-dispersive} or Taylor~\cite{TaylorPDE,TaylorIII}.

\begin{lemma}\label{lemm:dec dya 2}
Let $n\ge 1$.
There exist $\psi \in C_0^\infty(\xR^d)$ and
$\varphi \in C_0^\infty(\xR^d)$ such that the following properties are satisfied:
\begin{enumerate}
\item we have $0\le \psi\le 1$, $0\le \varphi\le 1$ and
\[
\supp \psi \subset\{\la\xi\ra \le 1\},
\qquad \supp \varphi\subset \Bigl\{\frac34\le |\xi| \le 2 \Bigr\};
\]
\item for all $\xi\in\xR^n$,
\be\label{17:1325n1}
1=\psi(\xi) +\sum_{p=0}^\infty \varphi(2^{-p}\xi);
\ee

\item (almost orthogonality) for all $\xi\in\xR^n$,
\be\label{i:ao}
\frac13\le\psi^2(\xi)+\sum_{p=0}^{+\infty}\varphi^2(2^{-p}\xi)\le 1.
\ee
\end{enumerate}
\end{lemma}
\begin{proof}
Consider a function 
$\psi \in C^{\infty}_{0}(\xR^d;\xR)$
satisfying
$\psi (\xi) = 1$ for $\la\xi\ra\le 3/4$, and
$\psi (\xi)=0$ for $\la \xi \ra \ge 1$. 
Next, we define
$\varphi (\xi)= \psi (\xi /2) - \psi(\xi)$ and we note that~$\varphi$ is supported in the annulus
$\{3/4 \le \la \xi \ra \le 2\}$.
For all integers~$N$ and all $\xi \in\xR^d$, we have
\[
\psi(\xi)+\sum_{p=0}^N \varphi(2^{-p}\xi)=\psi(2^{-N-1}\xi),
\]
which immediately implies \eqref{17:1325n1} by letting
$N$ tend to $+\infty$.

It remains to prove~\eqref{i:ao}. For all integers $N$, we have
\[
\psi^2 (\xi) + \sum_{p=0}^{N}
\varphi^2 (2^{-p} \xi)\le \Bigl(\psi (\xi) + \sum_{p=0}^{N}\varphi(2^{-p} \xi)\Bigr)^2.
\]
On the other hand, we note that, for all $\xi\in \xR^d$, there are never more than three non-zero terms in
the set $\{\psi(\xi),~\varphi(\xi),\ldots,\varphi(2^{-p}\xi),\ldots\}$.
Therefore, using the elementary inequality $(a+b+c)^2\le 3(a^2+b^2+c^2)$, we obtain
\[
\Bigl(\psi (\xi) + \sum_{p=0}^{N}\varphi(2^{-p} \xi)\Bigr)^2\le
3\Bigl(\psi^2 (\xi) + \sum_{p=0}^{N}\varphi^2(2^{-p}\xi)\Bigr).
\]
We then obtain~\eqref{i:ao} by letting $N$ tend to $+\infty$ in the previous inequalities.
\end{proof}

Let us define, for $p \ge 0$, the Fourier multipliers~$\Delta_{p}$ as follows:
\[
\Delta_{0} := \psi (D_x) \quad\text{and}\quad
\Delta_{p} := \varphi\bigl(2^{-(p-1)}D_x\bigr) \quad(p\ge 1).
\]
Then we obtain a decomposition of the identity: we have
\[
I = \sum_{p \ge 0} \Delta_{p},
\]
in the sense of distributions: for all $u\in \Sr'(\xR^d)$,
the series $\sum \Delta_{p}u$
converges to~$u$ in $\Sr'(\xR^d)$, which means that $\sum_{p=0}^N\langle \Delta_{p}u,\varphi\rangle_{\Sr'\times \Sr}$ converges to $\langle u,\varphi\rangle_{\Sr'\times \Sr}$ when $N$ goes to $+\infty,$  for all $\varphi\in \Sr(\xR^d)$.

We then recall how to 
use the Littlewood--Paley decomposition to characterize Sobolev spaces.

\begin{proposition}
Consider $s\in\xR$. A tempered distribution $u\in \Sr'(\xR^n)$ belongs to the Sobolev space $H^s(\xR^n)$ if and only if
\begin{enumerate}
\item[(a)] for all $p\ge 0$, $\Delta_pu\in L^2(\xR^n)$;
\item[(b)] the sequence $\delta_p=2^{ps}\|\Delta_p u\|_{L^2}$ belongs to $\ell^2(\xN)$.
\end{enumerate}
Moreover, there exists a constant~$C$ such that
\be\label{i:aoHs}
\frac{1}{C}\|u\|_{H^s}\le \Bigl(\sum_{p=0}^{+\infty} \delta_p^2\Bigr)^{\mez}\le C\|u\|_{H^s}.
\ee
\end{proposition}

\begin{proof}
It follows from~\eqref{i:ao} and the Plancherel identity that
\be\label{i:aoL2}
\sum_{p\ge 0} \|\Delta_p u\|_{L^2}^2
\le \lA u \rA_{L^2}^2
\le 3\sum_{p\ge 0}\|\Delta_p u\|_{L^2}^2.
\ee
Denote by $\langle D_x\rangle^s$ the Fourier multiplier $(I-\Delta)^{s/2}$. 
Since $\|u\|_{H^s}=\lA \langle D_x\rangle^s u\rA_{L^2}$, by applying \eqref{i:aoL2} with
$u$ replaced by $\langle D_x\rangle^s u$, we get
\[
\sum_{p\ge 0} \lA \Delta_p \langle D_x\rangle^s u\rA_{L^2}^2
\le \lA u \rA_{H^s}^2
\le 3\sum_{p\ge 0}\lA \Delta_p \langle D_x\rangle^su\rA_{L^2}^2.
\]
Consider $p\ge 1$ and write
\[
\lA \Delta_p \langle D_x\rangle^s u\rA_{L^2}^2=(2\pi)^{-n}\int_{\xR^n}(1+|\xi|^2)^s \varphi^2(2^{-p}\xi)\vert\hat{u}(\xi)\vert^2d\xi.
\]
Since $(1+|\xi|^2)^s \varphi^2(2^{-p}\xi) \sim 2^{2ps}$ on the support of $\varphi^2(2^{-p}\xi)$, we see that
\be\label{est:Hsdyadique}
\frac{1}{C}\,2^{2ps}\,\|\Delta_p u\|_{L^2}^2\le \lA \Delta_p \langle D_x\rangle^s u\rA_{L^2}^2\le C 2^{2ps}\,\|\Delta_p u\|_{L^2}^2,
\ee
for a certain constant~$C$ depending only on $s$. We have a similar estimate for $\Delta_{0}u$ and the desired result follows easily.
\end{proof}

\end{document}